\documentclass[a4paper,11pt,twoside]{amsart}
\usepackage{amssymb}
\usepackage[all]{xy}
\newcommand{\R}{\mathbb{R}}

\newcommand{\Z}{\mathbb{Z}}
\newcommand{\Q}{\mathbb{Q}}
\newcommand{\C}{\mathbb{C}}
\newcommand{\sheaf}{\mathcal{O}}
\newcommand{\liealg}[1]{\mathfrak{#1}}
\newcommand{\skal}[2]{\langle{#1},{#2}\rangle}

\newtheorem{theorem}{Theorem}[section]
\newtheorem{lemma}[theorem]{Lemma}
\newtheorem{proposition}[theorem]{Proposition}
\newtheorem{corollary}[theorem]{Corollary}

\theoremstyle{definition}
\newtheorem{definition}[theorem]{Definition}
\newtheorem{example}[theorem]{Example}

\theoremstyle{remark}
\newtheorem{remark}[theorem]{Remark}

\numberwithin{equation}{section}

\begin{document}
\title[Lorentzian manifolds with indecomposable holonomy groups]{A class of Lorentzian manifolds with indecomposable holonomy groups}
%    author one information
\author{Kordian L\"{a}rz}
\address{Humboldt-Universit\"{a}t Berlin, Institut f\"{u}r Mathematik R. 1.305, Rudower Chaussee 25, 12489 Berlin}
%\curraddr{}
\email{laerz@math.hu-berlin.de}
%\thanks{}
%
%    author two information
%\curraddr{}
%\thanks{}
%
%    \subjclass is required.
%
%\date{}
%\dedicatory{}
%
%    The "communicated by" line appears only in ERA, PROC and JAG.
%\commby{}
%
%    Abstract is required.
\begin{abstract}
	We consider a class of $S^{1}$-bundles whose total space admits a nowhere vanishing recurrent lightlike vector field with
	respect to a Lorentzian metric. This metric can be modified such that its restricted holonomy group is indecomposable
	and reducible. We apply Hodge theory to construct examples with Hermitian screen holonomy. Finally, we construct complete pp-waves.
\end{abstract}
\maketitle
\section{Introduction}
For a Lorentzian manifold $(X,g)$ of dimension $n+2$ let $Hol^{0}(X,g)$ be the connected component of its holonomy group.
By Wu's theorem \cite{MR867684} $(X,g)$ is locally a product of semi-Riemannian manifolds if its holonomy representation decomposes. Therefore, we may
focus on Lorentzian manifolds with indecomposable restricted holonomy groups. In contrast to the positive definite case
$Hol^{0}(X,g) \subset SO_{0}(1,n+1)$ does not need to be irreducible if it is indecomposable. In fact $SO_{0}(1,n+1)$ is the only connected
irreducible subgroup of $SO_{0}(1,n+1)$ (see \cite{MR1836778}).\par
The action of a reducible indecomposable subgroup of $SO_{0}(1,n+1)$ on $\R^{1,n+1}$ leaves a degenerate subspace $W$ invariant and we
get an invariant lightlike line $W \cap W^{\perp}$. Under the action of $Hol^{0}(X,g)$ this corresponds locally to a lightlike subbundle
$\Xi \subset TX$ of rank one which is spanned by a non-vanishing recurrent lightlike vector field\footnote{Here
we say that a vector field $V$ is recurrent if $\nabla_{\cdot}V = \alpha \otimes V$ for some 1-form $\alpha$ where $\nabla$ is the
Levi-Civita connection of $(X,g)$.}.
If $v$ is a lightlike vector in $\R^{1,n+1}$ spanning the invariant line, then
$Hol^{0}(X,g) \subset Stab(\R \cdot v) \subset SO_{0}(1,n+1)$. It can be shown that
$Stab(\R \cdot v) \cong (\R^{*} \times SO(n)) \ltimes \R^{n}$. If we choose a basis $(v,e_{1},\ldots,e_{n},z)$ of $\R^{n+2}$
satisfying $g(e_{i},e_{j})=\delta_{ij}$, $g(v,z)=1$ and $g(v,v)=g(z,z)=0$ then the Lie algebra
$(\R \oplus \liealg{so}(n)) \ltimes \R^{n}$ of $(\R^{*} \times SO(n)) \ltimes \R^{n}$ is given by
\begin{equation*}
	\left\lbrace \begin{pmatrix} a & w^{T} & 0\\ 0 & A & -w\\ 0 & 0 & -a \end{pmatrix}
								: a \in \R,~A \in \liealg{so}(n),~w \in \R^{n} \right\rbrace.
\end{equation*}
Being a subalgebra of a compact Lie algebra the projection
\begin{equation*}
	\liealg{g} := pr_{\liealg{so}(n)}(\liealg{h}) \subset \liealg{so}(n)
\end{equation*}
is compact and therefore reductive. Hence, $\liealg{g} = \liealg{z}(\liealg{g})+[\liealg{g},\liealg{g}]$ where
$\liealg{z}(\liealg{g})$ is the center of $\liealg{g}$.
\begin{theorem}[B{\'e}rard-Bergery \& Ikemakhen \cite{MR1216527}, Leistner \cite{MR2331527}]\label{BBILthm}
	Let $\liealg{h}$ be an indecomposable Lorentzian holonomy algebra.
	\begin{enumerate}
		\item \noindent
		Then $\liealg{h}$ belongs to one of the following types:
		\begin{itemize}
			\item \noindent
			Type 1: $\liealg{h} = (\R \oplus \liealg{g}) \ltimes \R^{n}$
			\item \noindent
			Type 2: $\liealg{h} = \liealg{g} \ltimes \R^{n}$
			\item \noindent
			Type 3:
			\begin{equation*}
				\liealg{h} = \left\lbrace \begin{pmatrix}
								\varphi(A) & w^T & 0\\
								0 & A & w\\
								0 & 0 & -\varphi(A)
							\end{pmatrix}: A \in \liealg{g},~w \in \R^{n} \right\rbrace
			\end{equation*}
			where $\varphi: \liealg{g} \twoheadrightarrow \R$ is an epimorphism satisfying
			$\varphi|_{[\liealg{g},\liealg{g}]}=0$.
			\item \noindent
			Type 4:
			There is $0< \ell < n$ such that $\R^{n} = \R^{\ell} \oplus \R^{n-\ell}$,
			$\liealg{g} \subset \liealg{so}(\ell)$ and
			\begin{equation*}
				\liealg{h} = \left\lbrace \begin{pmatrix}
								0 & \psi(A)^T & w^T & 0\\
								0 & 0 & 0 & -\psi(A)\\
								0 & 0 & A & -w\\
								0 & 0 & 0 & 0
							\end{pmatrix}: A \in \liealg{g},~w \in \R^{\ell}  \right\rbrace
			\end{equation*}
			for some epimorphism $\psi: \liealg{g} \twoheadrightarrow \R^{n-\ell}$ satisfying
			$\psi|_{[\liealg{g},\liealg{g}]}=0$.
		\end{itemize}
		\item \noindent
		The projection $\liealg{g}=pr_{\liealg{so}(n)}(\liealg{h})$ is the holonomy algebra of a Riemannian manifold.\qed
	\end{enumerate}
\end{theorem}
All possible holonomy groups can be constructed by Lorentzian metrics on $\R^{n+2}$ (see \cite{MR2264404}). The problem
to construct topologically nontrivial examples with certain properties is largely open.
A first approach has been made in \cite{MR2350042} where non-trivial globally hyperbolic examples of the form $X=\R^{2} \times M^{n}$
have been derived. Recently, globally hyperbolic examples for most Lorentzian holonomies have been found in \cite{bazaikin-2009}.
\par \bigskip
In this paper we study a class of Lorentzian metrics on the total space of $S^1$-bundles whose holonomy
group is reducible but indecomposable. All manifolds are assumed to be connected without boundary.
\section{Lorentzian manifolds with indecomposable holonomy group}
\subsection{Global properties}\hfill \par
In the following let $(X,g)$ be a Lorentzian manifold whose (full) holonomy group is indecomposable and reducible.
The screen bundle $\mathcal{S}$ of $(X,g)$ is defined as $\mathcal{S} = \mbox{Coker}(\Xi \hookrightarrow \Xi^{\perp})$. From
$\nabla^{(X,g)}$ we derive the induced connection $\nabla^{\mathcal{S}}$ on $\mathcal{S}$. Since
$Hol(X,g) \subset (\R^{*} \times O(n)) \ltimes \R^{n}$ we may define $G:=pr_{O(n)}(Hol(X,g))$. It can be shown that 
$Hol(\mathcal{S},\nabla^{\mathcal{S}}) = G$ and $\liealg{hol}(\mathcal{S},\nabla^{\mathcal{S}})=\liealg{g}$ \cite{MR2435292}.\par
In order to study the geometry of $(\mathcal{S},\nabla^{\mathcal{S}})$ it is convenient to study a non-canonical realization of $\mathcal{S}$ as a
distribution in $TX$ given by a non-canonical splitting $s$ of the exact sequence
\begin{equation*}
	\xymatrix{ 0 \ar[r] & \Xi \ar[r] & \Xi^{\perp} \ar[r] & \mathcal{S} \ar@/_/[l]_{s} \ar[r] & 0}.
\end{equation*}
Hence, we may define $S:= s(\mathcal{S})$. Since $\Xi \subset S^{\perp}$ there is a uniquely defined isotropic distribution $\Theta \subset S^{\perp}$ of rank
one with the following property:
If $V \in \Gamma(U\subset X,\Xi)$ then there exists $Z \in \Gamma(U\subset X,\Theta)$ such that $g(V,Z)=1$.\par
The Levi-Civita connection on $(X,g)$ induces connections\footnote{The relation of the curvatures of $\nabla^{\Xi}$ and $\nabla^{S}$ to the holonomy of $(X,g)$
												has been studied in \cite{bezvitnaya-2005}.} on the subbundles $\Xi$ and $S$ given by
\begin{equation*}
	\nabla^{\Xi} := pr_{\Xi} \circ \nabla|_{\Xi} \qquad \text{and} \qquad \nabla^{S} := pr_{S} \circ \nabla|_{S}.
\end{equation*}
Moreover, the canonical bundle morphism $S \stackrel{F}{\longrightarrow} \mathcal{S}$ is easily shown to be an isomorphism such that
$\nabla^{S}=F^{*}\nabla^{\mathcal{S}}$, i.e., $Hol(S,\nabla^{S})=G$.
\begin{definition}
	A manifold $X^{n}$ is totally twisted if there is no homeomorphism $X \rightarrow \R \times Y$ or $X \rightarrow S^{1} \times Y$
	where $Y$ is of dimension $n-1$.
\end{definition}
In the next section we will construct totally twisted Lorentzian manifolds whose (full) holonomy representation is indecomposable and reducible such that
$Hol^{0}(X,g)$ is of type 1 or 2. In the non-compact case the first Betti number for these manifolds will be zero.
\subsection{Local properties}\hfill \par
If $Hol^{0}(X,g)$ is indecomposable and reducible we have a locally defined recurrent lightlike vector field $V$ around $p \in X$. It is shown in
\cite{MR0035085} that we can find local coordinates $(x,y^{1},\ldots,y^{n},z)$ in $U \ni p$ such that
\begin{equation*}
	g = 2dxdz + u_{i}dy^{i}dz + fdz^{2} + g_{\alpha\beta}dy^{\alpha}dy^{\beta}
\end{equation*}
and $\frac{\partial}{\partial x} \in \Xi$ on $U$ where $u_{i},f \in C^{\infty}(U)$ and
$\frac{\partial u_{i}}{\partial x}=\frac{g_{\alpha\beta}}{\partial x}=0$.
Local coordinates of this form will be called Walker coordinates.
\begin{proposition}
	Let $(X,g)$ be a Lorentzian manifold such that $\liealg{hol}^{loc}_{p}(X,g)$ indecomposable and reducible for all $p \in X$.
	Then
	\begin{enumerate}
		\item
		$Hol(X,g)$ is indecomposable and reducible.
		\item
		$\Xi$ admits a global nowhere vanishing section if and only if $(X,g)$ is time-orientable.
		\item
		$Hol^{0}(X,g)$ is of type 2 or 4 if and only if there is a $\liealg{hol}^{loc}_{p}(X,g)$-invariant non-zero vector for all $p \in X$.
	\end{enumerate}
\end{proposition}
\begin{proof}
	Since $\liealg{hol}^{loc}_{p}(X,g)=\liealg{hol}_{p}(U_{\alpha},g|_{U_{\alpha}})$ for some neighborhood $U_{\alpha} \ni p$ we have
	$\Xi^{U_{\alpha}}|_{U_{\alpha}\cap U_{\beta}} = \Xi^{U_{\beta}}|_{U_{\alpha} \cap U_{\beta}}$,
	i.e., there is a $Hol(X,g)$-invariant isotropic distribution on $X$.\par
	If $(X,g)$ is time-orientable we may locally choose future pointing sections $V_{\alpha} \in \Gamma(U_{\alpha},\Xi)$ and use a partition of
	unity. Conversely, if $\Xi$ admits a global nowhere vanishing section $V$, so does $\Theta$. If $Z \in \Gamma(X,\Theta)$ denotes this section
	then $\frac{1}{\sqrt{2}}(V-Z)$ is a timelike unit vector field.\par
	For the last statement assume $\nabla_{\cdot}V^{U_{\alpha}}=0$ for some local sections $V^{U_{\alpha}} \in \Gamma(U_{\alpha},\Xi)$.
	If $V \in \Gamma(X,\Xi)$ is nowhere vanishing then $V|_{U_{\alpha}}=\lambda^{U_{\alpha}}V^{U_{\alpha}}$ and
	\begin{equation*}
		\nabla_{\cdot}V|_{U_{\alpha}} = d(\log(\lambda^{U_{\alpha}}))(\cdot)\lambda^{U_{\alpha}}V^{U_{\alpha}}
						= d(\log(\lambda^{U_{\alpha}}))(\cdot)V|_{U_{\alpha}}.
	\end{equation*}
	By construction $d(\log(\lambda^{U_{\alpha}}))=d(\log(\lambda^{U_{\beta}}))$ on $U_{\alpha} \cap U_{\beta}$, i.e.,
	$\nabla_{\cdot}V = \alpha(\cdot)V$ for some closed $1$-form $\alpha$.
\end{proof}
However, the existence of a covering of $X$ by indecomposable Walker coordinates does not imply reducibility
of $Hol^{0}(X,g)$.\footnote{A counterexample can be constructed as follows: Let $f_{1},f_{2} \in C^{\infty}(\R)$ such that
				$f_{1}|_{]-\infty,-1]}=f_{2}|_{[1,\infty[}=1$ and $f_{1}|_{[-\frac{1}{2},\infty[}=f_{2}|_{]-\infty,\frac{1}{2}]}=0$.
				On $\R^{3}$ we define $g=2dxdz + y^{2}f_{1}(z)dz^{2} + y^{2}f_{2}(z)dx^{2} + dy^{2}$.
				Then $Hol^{0}(\R^{3},g)=SO_{0}(1,2)$.}
For any given Walker coordinates an integrable realization of the screen bundle is given by $S:= \mbox{span}\{\frac{\partial}{\partial y^{\alpha}}\}$.
In this case $\Xi = \mbox{span}\{\frac{\partial}{\partial x}\}$ and $\Theta:= \mbox{span}\{Z\}$ with
$Z:= \frac{1}{2}(g^{\alpha\beta}u_{\alpha}u_{\beta} -f)\frac{\partial}{\partial x}
		-g^{\alpha\beta}u_{\alpha}\frac{\partial}{\partial y^{\beta}} + \frac{\partial}{\partial z}$. In particular, the parallel
transport equations immediately imply $Hol(M_{xz},\nabla^{M_{xz}}) \subset Hol(S,\nabla^{S})$ where $M_{xz}$ is the Riemannian submanifold
given by the $\frac{\partial}{\partial y^{\alpha}}$-coordinates.\footnote{Note however, that all indecomposable Lorentzian holonomies have been realized
											in \cite{MR2264404} by Walker coordinates for which
											$Hol(M_{xz},\nabla^{M_{xz}})=0$.}

\section{The total space of an $S^{1}$-bundle as a Lorentzian manifold}
First, we will construct a Lorentzian metric on the total space of an $S^{1}$-bundle over a base manifold admitting a nowhere vanishing closed 1-form.
Under this metric the vertical vector field on the total space becomes recurrent. Then we will give conditions under which the restricted holonomy
representation becomes indecomposable. The idea is based on the following well known observation:
The exact sequence of sheaves
\begin{equation*}
	0 \longrightarrow \Z \longrightarrow \mathcal{C}^{\infty}_{M} \stackrel{\exp}{\longrightarrow}
	{\mathcal{C}^{\infty}_{M}}^{*} \longrightarrow 0
\end{equation*}
provides the isomorphism
\begin{equation*}
	c_{1}: \{\mbox{iso. classes of complex line bundles on }M \}  \rightarrow  H^2(M,\Z).
\end{equation*}
As usual we write $[\frac{\psi}{2\pi}] \in H^{2}(M,\Z)$ if $\psi$ is a closed 2-form and
$[\frac{\psi}{2\pi}] \in \mbox{Im}(H^{2}(M,\Z) \rightarrow H^{2}_{dR}(M,\R))$.
We state the main construction method:
\begin{proposition}\label{mainconstr}
	Let $(M,g)$ be a Riemannian manifold and $\eta$ a nowhere vanishing closed 1-form on $M$. Let $\psi$ be a
	2-form on $M$ with $[\frac{\psi}{2\pi}] \in H^2(M,\Z)$. Then there exists an $S^{1}$-bundle $\pi: X \longrightarrow M$
	satisfying $c_{1}(X \rightarrow M) = [\frac{\psi}{2\pi}]$ and
	\begin{enumerate}
		\item \noindent
		There is a global non-vanishing 1-form $\theta$ on $X$ such that
		\begin{equation*}
			\tilde{g}_{f} := 2\theta \pi^{*}\eta + f \cdot \pi^{*}\eta^{2} + \pi^{*}g
		\end{equation*}
		defines a Lorentzian metric on $X$ for any $f \in C^{\infty}(X)$.
		\item \noindent
		Given $p \in X$ and a local 1-form $\phi$ with $\psi = d\phi$ there are local coordinates
		$(x,y^{1},\ldots,y^{n},z)$ around $p$ such that
		\small{
		\begin{equation*}
			\qquad \tilde{g}_{f} = 2dxdz + (u_{i} + 2g_{i(n+1)})dy^{i}dz + (f+\frac{u_{n+1}}{2}+g_{(n+1),(n+1)})dz^{2} + g_{ij}dy^{i}dy^{j}
		\end{equation*}}
		where $2\phi = u_{i}dy^{i}+u_{n+1}dz$.
		\item \noindent
		The $U(1)$-action of $X \rightarrow M$ acts by isometries on $(X,\tilde{g}_{f})$ if $f$ is constant on the fibers.
		\item \noindent
		The vertical vector field is a global lightlike vector field which is parallel if and only if
		$f$ is constant on the fibers.
	\end{enumerate}
\end{proposition}
\begin{proof}
	Consider the smooth complex line bundle $L \rightarrow M$ given by $c_{1}^{-1}(-[\frac{\psi}{2\pi}])$ and some Hermitian
	metric $h$ on $L$. The curvature endomorphism of the Chern connection $\nabla_{h}^{C}$ of $(L,h)$ is given by the closed
	imaginary $(1,1)$-form $iF_{h}$. Moreover, $[iF_{h}]=-2\pi i c_{1}(L)$. Hence, $F_{h}-\psi = df$ is exact and
	$\nabla_{L} := \nabla_{h}^{C} - if$ is another Hermitian connection on $L$ with respect to $h$ and its curvature
	endomorphism is given by $i\psi$.\par
	The metric $h$ provides a $U(1)$-reduction of the $GL(1,\C)$-bundle $(L,h)$. Since $\nabla_{L}$ is Hermitian it reduces as
	well. In this way, we derive an $S^{1}$-bundle $X:= \{ v \in L: h(v,v)=1 \} \rightarrow M$ together with the $U(1)$-connection
	$\nabla_{L}$.\par
	Consider the $1$-form $\eta$ on $M$. By Frobenius' theorem we can find for all $x \in M$ local coordinates
	$(y_{1},\ldots,y_{n},z)$ on some neighborhood $U \ni x$ such that $\eta = dz$. Moreover, we may assume that
	$X \rightarrow M$ is trivial over $U$ and $\psi = d\phi_{U}$. Consider a unit length section $s_{U}: U \rightarrow L$ such
	that\footnote{We can find such a section without any restriction: If $t:U \rightarrow L$ is any unit length section
	we have $\nabla_{L}t=i\alpha \otimes t$ for some 1-form $\alpha$. Hence, $\alpha -\phi_{U} = df$ and $s_{U}:=e^{-if}t$ has
	the connection form $i\phi_{U}$.}
	\begin{equation*}
		\nabla_{L}s_{U} = i\phi_{U} \otimes s_{U}.
	\end{equation*}
	Using the section $s_{U}$ we may define local coordinates $(x^{0},\ldots,x^{n+1}):=(x,y^{1},\ldots,y^{n},z)$ given by
	$e^{ix}s_{U}(y_{1},\ldots,y_{n},z)$.\par
	We have to construct the 1-form $\theta$ from the statement. For this we consider another coordinate neighborhood
	$V \subset M$. Assume we have a 1-form $\phi_{V}$ on $V$ such that $d\phi_{V} = \psi$ and a local unit length
	section $s_{V}:V \rightarrow L$ such that $\nabla_{L}s_{V}=i\phi_{V} \otimes s_{V}$. If $U \cap V \neq \emptyset$ we
	have $s_{V} = e^{ig_{UV}}s_{U}$. Therefore
	\begin{equation*}
		\nabla^{L}s_{V} = \nabla^{L}(e^{ig_{UV}}s_{U}) = idg_{UV} \otimes s_{V} + i\phi_{U} \otimes s_{V}
				=i(dg_{UV} + \phi_{U})s_{V}
	\end{equation*}
	and we conclude $dg_{UV}=\phi_{V} - \phi_{U}$. Given the local coordinates defined by $s_{U}$ and $s_{V}$ we observe
	\begin{equation*}
		e^{ix_{U}}s_{U} = e^{i(x_{V} +c)}s_{V} = e^{i(x_{V}+g_{UV}+c)}s_{U},
	\end{equation*}
	i.e., $dx_{U} - dx_{V} = dg_{UV} = \phi_{V}-\phi_{U}$. From this equation we conclude that $dx_{U}+\phi_{U}$ glues to a
	global non-vanishing 1-form $\theta$ on $X$. Moreover, $d\theta =\pi^{*}\psi$, i.e., the pullback of $\psi$ is an exact
	form on $X$.\par
	We have to show that $\tilde{g}_{f}:=2\theta \pi^{*}\eta + f \cdot \pi^{*}\eta^{2} + \pi^{*}g$ is a Lorentzian metric. This can be
	checked in the given local coordinate expression
	\begin{equation*}
		\tilde{g}_{f} = \begin{pmatrix}
				0 & 0 & \cdots & 0 & 1\\
				0  & g_{11} & \cdots & g_{1n} & g_{1(n+1)}+u_{1}\\
				\vdots & \vdots & \ddots & \vdots & \vdots\\
				0 & g_{n1} & \cdots & g_{nn} & g_{n(n+1)}+u_{n}\\
				1 & g_{(n+1)1}+u_{1} & \ldots & g_{(n+1)n}+u_{n} & f+u_{n+1}+g_{(n+1)(n+1)}
		            \end{pmatrix}
	\end{equation*}
	and we conclude $\det(\tilde{g}_{ij})<0$ since $(g_{ij})_{1 \leq i,j \leq n}$ is the Riemannian metric $g$ restricted to the
	submanifold $\{(y_{1},\ldots,y_{n},\text{const.})\}$.\par
	If $f \in C^{\infty}(X)$ is constant on the fibers the $U(1)$-action of the bundle leaves $\tilde{g}_{f}$ invariant since
	$\theta$ is the connection 1-form of $\nabla_{L}$ and all other terms in $\tilde{g}_{f}$ are pullbacks.\par
	By definition of $\tilde{g}_{f}$ the vertical vector field is lightlike.
	Using the local coordinate expression for $\tilde{g}_{f}$ we compute
	\begin{equation*}
		\Gamma_{0i}^{k} =\frac{1}{2}\delta_{i(n+1)}\delta_{k0}\frac{\partial f}{\partial x^{0}} \quad i,k \in \{0,\ldots,n+1\}.
	\end{equation*}
	Therefore, the vertical vector field is parallel if and only if $f$ is constant on the fibers.
\end{proof}
Up to diffeomorphism $X$ depends only on the choice of the class $[\frac{\psi}{2\pi}] \in H^{2}(M,\Z)$. However, the Lorentzian metric $g_{f}$ depends on the particular representative $\psi \in [\psi] \in H^{2}(M,\R)$. This will be important in the following sections.
\begin{proposition}\label{exclty3}
	Assume $(X,\tilde{g}_{f}) \rightarrow (M,g)$ is constructed as in Proposition $\ref{mainconstr}$ such that $(X,\tilde{g}_{f})$ is indecomposable.
	Then $(X,\tilde{g}_{f})$ is of type 1 if and only if $\frac{\partial f}{\partial x}|_{p} \neq 0$ for some $p \in X$.
	In particular, $(X,\tilde{g}_{f})$ is not of type 3.
\end{proposition}
\begin{proof}
	Consider the vertical line subbundle $\Xi$ of $TX$ spanned by $\frac{\partial}{\partial x}$ and its induced connection
	$\nabla^{\Xi}$. Write $R^{\nabla^{\Xi}}$ for the curvature of $\nabla^{\Xi}$. Given the local coordinates
	$(x^{0},x^{1},\ldots,x^{n},x^{n+1}):=(x,y^{1},\ldots,y^{n},z)$ from Proposition $\ref{mainconstr}$ we know
	\begin{equation*}
		\tilde{g}_{f} = 2dxdz + \tilde{u}_{i}dy^{i}dz + \tilde{f} \cdot dz^{2} + g_{ij}dy^{i}dy^{j}.
	\end{equation*}
	Using $\Gamma_{0i}^{k}=\frac{1}{2}\delta_{i(n+1)}\delta_{k0}\frac{\partial \tilde{f}}{\partial x^{0}}$ a short computation shows
	\begin{equation*}
		R^{\nabla^{\Xi}}(\frac{\partial}{\partial x^{i}},\frac{\partial}{\partial x^{j}})\frac{\partial}{\partial x^{0}}
			= \frac{1}{2}(\delta_{j(n+1)}\frac{\partial^{2}\tilde{f}}{\partial x^{i} \partial x^{0}}
			-\delta_{i(n+1)}\frac{\partial^{2}\tilde{f}}{\partial x^{j} \partial x^{0}})\frac{\partial}{\partial x^{0}}.
	\end{equation*}
	It is shown in \cite{bezvitnaya-2005} that $\liealg{hol}(X,\tilde{g}_{f})$ is of type 2 or 4 if and only if $R^{\nabla^{\Xi}} = 0$.
	Using the formula for $R^{\nabla^{\Xi}}$ we conclude that $(X,\tilde{g}_{f})$ is not of type 1 or type 3 if
	$\frac{\partial \tilde{f}}{\partial x}=0$. Assume
	$0 = R^{\nabla^{\Xi}}(\frac{\partial}{\partial x^{0}},\frac{\partial}{\partial x^{n+1}})\frac{\partial}{\partial x}
				= \frac{1}{2}\frac{\partial^{2}\tilde{f}}{\partial(x^{0})^{2}}$.
	If $\tilde{f}$ is restricted to a fiber $S^{1}$ then we conclude that $\tilde{f}$ is constant on the fiber, i.e.,
	$\frac{\partial\tilde{f}}{\partial x^{0}}=0$. In particular, $R^{\nabla^{\Xi}}=0$ if and only if
	$\frac{\partial\tilde{f}}{\partial x^{0}}=0$. Hence, $(X,\tilde{g}_{f})$ is of type 1 or 3 if and only if
	$\frac{\partial f}{\partial x}|_{p} \neq 0$ for some $p \in X$. For the last statement we assume $(X,\tilde{g}_{f})$ is of type 3.
	In our case \cite{bezvitnaya-2005}[Prop. 6.2.1] implies
	$R^{\nabla^{\Xi}}(\frac{\partial}{\partial x^{0}},\frac{\partial}{\partial x^{n+1}})=0$ and therefore $R^{\nabla^{\Xi}}=0$.
	This is a contradiction.
\end{proof}
\begin{definition}
	If $(X,\tilde{g}_{f})$ is constructed as in Prop. \ref{mainconstr} we say $f \in C^{\infty}(X)$ is sufficiently generic if
	$\liealg{hol}(X,\tilde{g}_{f})$ is indecomposable and not of type 4.
\end{definition}
It is not difficult to construct sufficiently generic functions for Walker coordinates. Using a partition of unity we derive sufficiently generic functions on $X$.
\begin{example}
	Let $M:=\R^{3} \setminus \{(0,0,-1),(0,0,+1)\}$ and let $0 \neq [\frac{\psi}{2\pi}] \in H^2(M,\Z)$. Define $\eta:=\frac{\partial}{\partial z}$
	on $M$ and construct $(X,\tilde{g}_{f})$ as in Prop. \ref{mainconstr} with $f \in C^{\infty}(X)$ sufficiently generic. Then $X$ is totally twisted
	with $b_{1}(X)=0$ and $\liealg{hol}(X,\tilde{g}_{f})$ is of type 2 if $\frac{\partial f}{\partial x} \equiv 0$ and otherwise of type 1.
\end{example}
\begin{proof}
	From the long exact sequence of homotopy groups for the fibration $X \rightarrow M$ we conclude that $\pi_{1}(X)$ is a finite torsion group.
	Moreover, Gysin's sequence implies $H^{3}(X,\R)=\R^{2}$. Using the K\"{u}nneth formula we derive a contradiction unless $X$ is totally twisted.
\end{proof}
If $M = N \times L$ with $\dim L = 1$ and $N$ not necessarily compact we may define $\eta = \frac{\partial}{\partial x^{L}}$ and consider
$[\frac{\psi}{2\pi}] \in H^{2}(N,\Z)$. In this case $X = \tilde{X} \times L$ where $\tilde{X}$ is the total space of the bundle corresponding to
$[\frac{\psi}{2\pi}]$. Let $g$ be a Riemannian metric on $N$. Then we consider the Lorentzian metric
\begin{equation*}
	\tilde{g}_{f} := 2\theta dz + f \cdot dz^{2} + \pi^{*}g
\end{equation*}
for some function $f \in C^{\infty}(X)$, where $z$ is the coordinate on $L$. In this situation we say $(X,\tilde{g}_{f})$ is of
{\em toric type}\footnote{If $M=N \times S^{1}$ then $X$ is a torus bundle over $N$ where one direction in the fibers is trivial.}.
\begin{proposition}\label{screenhoriz}
	Let $(X,\tilde{g}_{f})$ be of toric type and $\pi:\tilde{X} \rightarrow M$ the corresponding $S^{1}$-bundle. Then:
	\begin{itemize}
		\item
		The horizontal distribution in $TX$ is isomorphic to the screen bundle,
		\item
		$Hol(M,g) \subset G$,
		\item
		$\liealg{hol}^{loc}_{\pi(p)}(M,g) \subset \liealg{hol}^{loc}_{(p,q)}(\mathcal{S},\nabla^{\mathcal{S}})$.
	\end{itemize}
\end{proposition}
\begin{proof}\par \noindent
	\begin{enumerate}
		\item \noindent
		We have $X = \tilde{X} \times L$ with $\dim L = 1$. If $\frac{\partial}{\partial z}$ is the global coordinate field on $L$
		and $V$ the vertical vector field of $\tilde{X}$ the screen bundle may be identified with
		$S= \mbox{span}\{V,Z:=-\frac{1}{2}fV+\frac{\partial}{\partial z}\}^{\perp}$. Since
		$\tilde{g}_{f} = 2dxdz + u_{i}dy^{i}dz + g_{ij}dy^{i}dy^{j} + fdz^{2}$ we observe
		\begin{equation*}
			\tilde{g}_{f}(V,\frac{\partial}{\partial y^{i}}-u_{i}\frac{\partial}{\partial x})
				= \tilde{g}_{f}(Z,\frac{\partial}{\partial y^{i}}-u_{i}\frac{\partial}{\partial x}) = 0,
		\end{equation*}
		i.e., $Y_{i}:= \frac{\partial}{\partial y^{i}} -u_{i}\frac{\partial}{\partial x} \in S$. However, a simple computation shows that the
		horizontal space of $\tilde{X}$ is spanned by $\{Y_{i}\}$.
		\item \noindent
		Fix $(p,q) \in X$ and let $x:=\pi(p)$. To each $a \in Hol_{x}(M,g)$ we construct a loop $\tilde{\gamma}:I \rightarrow X$ on which
		parallel displacement induces $a \in G$. More precisely, let $\gamma: [0,1] \rightarrow M$ be a loop with $\gamma(0)=x$ and let
		$\tilde{\delta}:[0,1] \rightarrow \tilde{X}$ with $\tilde{\delta}(0)=p$ be its horizontal lift with $u:= \tilde{\delta}(1)$.
		If $u \neq p$ let $\tilde{\beta}$ be the integral curve of the vertical field in the fiber $\pi^{-1}(x)$ connecting $u$ and $p$.
		We define
		\begin{equation*}
			\tilde{\gamma} := \begin{cases}
						(\tilde{\delta}*\tilde{\beta},q) & \text{if}~u \neq p,\\
						(\tilde{\delta},q) & \text{otherwise}.
			                  \end{cases}
		\end{equation*}
		Let $v \in T_{x}M$ and let $v_{t}$ be its parallel displacement along $\gamma$. Write $\tilde{v}_{t}$ for the horizontal lift of
		$v_{t}$. First, we consider the parallel displacement $w_{t}$ of $\tilde{v}_{0}$ along $\delta=(\tilde{\delta},q): I \rightarrow X$.
		We show $\tilde{v}_{t}=pr_{S}(w_{t})$. Clearly, the set $J \subset I$ on which this equation holds is non-empty and closed.
		In order to show that $J \subset I$ is open we may use local coordinates. For $1 \leq \alpha,\beta,k \leq n$ we have
		$w^{(n+1)}_{t}=\dot{\delta}^{(n+1)}_{t}=\Gamma^{k}_{\cdot 0} =0$ and
		\begin{equation*}
			0 = \dot{w}^{k}_{t} + \Gamma^{k}_{ij}\dot{\delta}^{i}_{t}w^{j}_{t}
			= \dot{w}^{k}_{t} + \Gamma^{k}_{\alpha\beta}\dot{\gamma}^{\alpha}_{t}w^{\beta}_{t}
			= \dot{w}^{k}_{t} + \tilde{\Gamma}^{k}_{\alpha\beta}\dot{\gamma}^{\alpha}_{t}w^{\beta}_{t},
		\end{equation*}
		where $\tilde{\Gamma}^{k}_{\alpha\beta}$ are the Christoffel symbols of $(M,g)$. This shows $w^{k}_{t}=v^{k}_{t}$, i.e.,
		$pr_{S}(w_{t})=\tilde{v}_{t}$.\par
		Assume $u \neq p$ and consider the parallel displacement of a vector $\tilde{v} \in \Xi^{\perp}_{u}$ along $\beta=(\tilde{\beta},q)$.
		Again we can work in a local coordinate chart and conclude $\tilde{v}^{k}_{t}=\mbox{const.}$
		\item
		Using the same arguments as above the last statement follows.
	\end{enumerate}
\end{proof}
\section{Examples and Hermitian screen bundles}
\begin{example}
	Let $(M,g)$ be a Riemannian manifold such that $\liealg{hol}(M,g)=\liealg{so}(n)$. If $(X,\tilde{g}_{f})$ is of toric type over $(M,g)$ and if
	$f \in C^{\infty}(X)$ is sufficiently generic then
	\begin{equation*}
		\liealg{hol}(X,\tilde{g}_{f}) = \begin{cases}
						\liealg{so}(n) \ltimes \R^{n} & \text{if}~\frac{\partial f}{\partial x} \equiv 0,\\
						(\R \oplus \liealg{so}(n)) \ltimes \R^{n} & \text{otherwise}.
		                            \end{cases}
	\end{equation*}
	\qed
\end{example}
In order to construct other examples we compute the curvature of $(S,\nabla^{S})$ in case that $(X,\tilde{g}_{f})$ is of toric type.
\begin{equation*}
	\nabla_{\partial_{i}}{Y_{j}} = (\Gamma^{k}_{ij}-u_{j}\Gamma^{k}_{0i})\frac{\partial}{\partial x^{k}}
						-\frac{\partial u_{j}}{\partial x^{i}}\frac{\partial}{\partial x}
\end{equation*}
implies $\nabla_{\partial_{0}}{Y_{j}} =-\frac{\partial u_{j}}{\partial x}\frac{\partial}{\partial x}=0$ and for $\alpha \in \{1,\ldots,n\}$
we have $\nabla^{S}_{Y_{i}}{Y_{j}}=pr_{S}(\nabla_{\frac{\partial}{\partial y^{i}}}{Y_{j}}) = \Gamma^{\alpha}_{ij}Y_{\alpha}$. Therefore,
$R^{\nabla^{S}}(Y_{i},Y_{j})Y_{k} = pr_{S}(R(\partial_{i},\partial_{j})\partial_{k}) = pr_{S}(R_{(M,g)}(\partial_{i},\partial_{j})\partial_{k})$.
Moreover, $R^{\nabla^{S}}(\partial_{0},Y_{i})Y_{j}=R^{\nabla^{S}}(\partial_{0},Z)Y_{k}=0$.\par
Any almost complex structure $J$ on $M$ induces an almost complex structure $\tilde{J}$ on $S$ since $J$ can be lifted to the horizontal
bundle. The same way we can lift other tensors to the screen bundle and conclude
\begin{lemma}\label{toric-holonomies}
	Let $(M,g)$ be a Riemannian manifold and $[\frac{\psi}{2\pi}] \in H^{2}(M,\Z)$. Let $(X = \tilde{X} \times L,g_{f})$ be of toric type
	where $\tilde{X} \rightarrow M$ is the $S^{1}$-bundle corresponding to $[\frac{\psi}{2\pi}]$ and $f \in C^{\infty}(X)$ is sufficiently generic.
	\begin{enumerate}
		\item \noindent
		$Hol^{0}(S,\nabla^{S})=0 \Leftrightarrow (M,g)$ is flat and $\nabla^{(M,g)}\psi =0$.
		\item \noindent
		If $(M,J,g)$ is K\"{a}hler then $\nabla^{S}\tilde{J}=0 \Leftrightarrow \psi \in \Lambda^{1,1}(M,J)$.
		\item \noindent
		If $(M,J,g)$ is K\"{a}hler with a parallel holomorphic volume form $\Omega$ then
		$\nabla^{S}\tilde{\Omega}=0 \Leftrightarrow \psi \in \Lambda^{1,1}(M,J)$ is a primitive form.\footnote{Remember, a 2-form $\psi$
						on $(M,J,g)$ is primitive if $\Lambda \psi =\sum_{i=1}^{\dim_{\C}(M,J)}{\psi(e_{i},Je_{i})} =0$ where
						$\Lambda$ is the dual Lefschetz operator.}
		\item \noindent
		If $(M,J_{1},J_{2},J_{3},g)$ is hyperk\"{a}hler then $\nabla^{S}\tilde{J}_{1}=\nabla^{S}\tilde{J}_{2}=\nabla^{S}\tilde{J}_{3}=0
		\Leftrightarrow \psi \in \Lambda^{1,1}(M,J_{1}) \cap \Lambda^{1,1}(M,J_{2})$.
		\item \noindent
		If $(M,\phi,g)$ is a $G_{2}$-manifold then $\nabla^{S}\tilde{\phi}=0 \Leftrightarrow \mathcal{BI}(C_{24}(\psi \otimes \phi)) =0$,
		where $C_{24}$ is the metric contraction over the second and the fourth slot and $\mathcal{BI}$ is the Bianchi projector.
		\item \noindent
		If $(M,\Omega,g)$ is a $Spin(7)$-manifold  then $\nabla^{S}\tilde{\Omega}=0 \Leftrightarrow \mathcal{AB}(C_{24}(\psi \otimes \Omega)) =0$,
		where $\mathcal{AB}$ is the alternating cyclic sum.
	\end{enumerate}
\end{lemma}
\begin{proof}\par \noindent
\begin{enumerate}
	\item \noindent
	By Proposition \ref{screenhoriz} $Hol^{0}(S,\nabla^{S})=0$ implies that $(M,g)$ is flat. Moreover, using local coordinates on $(M,g)$ such
	that $\Gamma^{\gamma}_{\alpha\beta}=0$ we have
	\begin{align*}
		R^{\nabla^{S}}(Y_{i},Z)Y_{k} &= \sum_{\alpha}\nabla^{S}_{Y_{i}}(\psi(\partial_{k},\partial_{\alpha})Y_{\alpha})
					= \sum_{\alpha}(Y_{i}(\psi(\partial_{k},\partial_{\alpha})))Y_{\alpha}\\
					&= \sum_{\alpha}((\nabla_{\partial_{i}}\psi)(\partial_{k},\partial_{\alpha}))Y_{\alpha},
	\end{align*}
	i.e. $R^{\nabla^{S}}=0 \Leftrightarrow R^{(M,g)}=0$ and $\nabla\psi =0$.
	\item \noindent
	If $(M,J,g)$ is K\"{a}hler let $(y^{1},\ldots,y^{2m})$ are local coordinates on $M$ such that $\partial_{2k}=J(\partial_{2k-1})$ and $Y_{j}$
	the horizontal lift of $\partial_{j}$. Since $\nabla^{(M,g)}J=0$ the only non-vanishing $\nabla^{S}_{\cdot}\tilde{J}$ can be 
	$\nabla^{S}_{Z}\tilde{J}$. However,
	\begin{equation*}
		\skal{\nabla^{S}_{Z}(\tilde{J}Y_{j})}{Y_{\ell}} = g^{k\alpha}\psi(J(\partial_{j}),\partial_{\alpha})\skal{Y_{k}}{Y_{\ell}}
							= \psi(J(\partial_{j}),\partial_{\ell})
	\end{equation*}
	and
	\begin{align*}
		\skal{\tilde{J}(\nabla^{S}_{Z}Y_{j})}{Y_{\ell}} &= -\skal{\nabla^{S}_{Z}Y_{j}}{\tilde{J}Y_{\ell}}
						= -g^{k\alpha}\psi(\partial_{j},\partial_{\alpha})\skal{Y_{k}}{\tilde{J}Y_{\ell}}\\
						&= -\psi(\partial_{j},J\partial_{\ell}),
	\end{align*}
	i.e., $\nabla^{S}_{Z}\tilde{J}=0 \Leftrightarrow \psi(J\partial_{j},\partial_{\ell})+\psi(\partial_{j},J\partial_{\ell}) =0$.
	\item \noindent
	We have to compute $\nabla_{Z}\tilde{\Omega}$. For $1\leq k \leq m$ define $Z_{k}:=\frac{1}{2}(Y_{k}-i\tilde{J}Y_{k})$. A short computation
	shows
	\begin{equation*}
		(\nabla_{Z}\tilde{\Omega})(Z_{1},\ldots,Z_{m}) = \sqrt{-1}(\Lambda \psi)\tilde{\Omega}(Z_{1},\ldots,Z_{m}),
	\end{equation*}
	i.e., $\nabla^{S}\tilde{\Omega}=0 \Leftrightarrow \Lambda \psi =0$.
	\item \noindent
	This follows from the second statement.
	\item \noindent
	If $(M,\phi,g)$ is a $G_{2}$-manifold with a parallel positive 3-form $\phi$ then
	\begin{align*}
		(\nabla^{S}_{Z}\tilde{\phi})(Y_{\alpha},Y_{\beta},Y_{\gamma}) 
%					&=
%					-g^{k\ell}(\psi(\partial_{\alpha},\partial_{\ell})\phi(\partial_{k},\partial_{\beta},\partial_{\gamma})\\
%						&\qquad +\psi(\partial_{\beta},\partial_{\ell})\phi(\partial_{\alpha},\partial_{k},\partial_{\gamma})\\
%						&\qquad +\psi(\partial_{\gamma},\partial_{\ell})\phi(\partial_{\alpha},\partial_{\beta},\partial_{k}))\\
					&= C_{24}(\psi \otimes \phi)(\partial_{\alpha},\partial_{\beta},\partial_{\gamma})\\
						&\qquad +C_{24}(\psi \otimes \phi)(\partial_{\beta},\partial_{\gamma},\partial_{\alpha})\\
						&\qquad +C_{24}(\psi \otimes \phi)(\partial_{\gamma},\partial_{\alpha},\partial_{\beta})\\
					&= \mathcal{BI}(C_{24}(\psi \otimes \phi))(\partial_{\alpha},\partial_{\beta},\partial_{\gamma}).
	\end{align*}
	\item \noindent
	If $(M,\Omega,g)$ is a $Spin(7)$-manifold with a parallel admissible 4-form $\Omega$ then
	\begin{align*}
		(\nabla^{S}_{Z}\tilde{\Omega})(Y_{\alpha},Y_{\beta},Y_{\gamma},Y_{\delta})
				&= C_{24}(\psi \otimes \Omega)(\partial_{\alpha},\partial_{\beta},\partial_{\gamma},\partial_{\delta})\\
					&\qquad -C_{24}(\psi \otimes \Omega)(\partial_{\beta},\partial_{\gamma},\partial_{\delta},\partial_{\alpha})\\
					&\qquad +C_{24}(\psi \otimes \Omega)(\partial_{\gamma},\partial_{\delta},\partial_{\alpha},\partial_{\beta})\\
					&\qquad -C_{24}(\psi \otimes \Omega)(\partial_{\delta},\partial_{\alpha},\partial_{\beta},\partial_{\gamma})\\
				&= \mathcal{AB}(C_{24}(\psi \otimes \Omega))(\partial_{\alpha},\partial_{\beta},\partial_{\gamma},\partial_{\delta}).
	\end{align*}
\end{enumerate}
\end{proof}
The lemma above provides sufficient conditions for a toric type Lorentzian manifold to have specified screen holonomy. For any complex manifold $X$ we
write
\begin{align*}
	H^{1,1}(X,\Z) &:= \mbox{Im}(H^{2}(X,\Z) \rightarrow H^{2}(X,\C)) \cap H^{1,1}(X),\\
	H^{1,1}(X,\Q) &:= \mbox{Im}(H^{2}(X,\Q) \rightarrow H^{2}(X,\C)) \cap H^{1,1}(X).
\end{align*}
If $X$ is a compact K\"{a}hler manifold the Lefschetz theorem on $(1,1)$-classes implies $H^{1,1}(X,\Z) = NS(X)$ where
$NS(X)$ is the Neron-Severi group of $X$ defined as the image of
\begin{equation*}
	\mbox{Pic}(X)=H^{1}(X,\sheaf^{*}_{X}) \stackrel{c_{1}}{\longrightarrow} H^{2}(X,\Z)
						\stackrel{\Z \subset \C}{\longrightarrow} H^{2}(X,\C).
\end{equation*}
By Lemma $\ref{toric-holonomies}$ we have $Hol(S,\nabla^{S}) \subset U(n)$ if $\psi \in \Lambda^{1,1}(M,J)$.
By definition $[\psi] \in H^{2}(M,\Z)$. Therefore $Hol(S,\nabla^{S}) \subset U(n)$ if $[\psi] \in NS(M,J)$. It is not difficult to construct examples
over $\C{P}^{n}$. In the non-symmetric case we may apply the following
\begin{corollary}\label{screen-unitary}
	Let $(M^{2n},J)$ be a compact simply-connected irreducible K\"{a}hler manifold with $c_{1}(M,J) < 0$ and\footnote{
											Explicit examples can be found in \cite{MR0451180}}
	$g$ its Einstein-K\"{a}hler metric. Let $\alpha \in H^{2}(M,\Z)$ be a Hodge class, e.g., $-c_{1}(M,J)$.
	If $(X=\tilde{X} \times L,\tilde{g}_{f})$ is of toric type over $(M,J,g)$ where $\tilde{X} \rightarrow M$ is
	constructed using a representative of $\alpha$ and if $f \in C^{\infty}$ is sufficiently generic then
	\begin{equation*}
		\liealg{hol}(X,\tilde{g}_{f}) = \begin{cases}
							\liealg{u}(n) \ltimes \R^{2n} &\text{if}~\frac{\partial f}{\partial x} \equiv 0,\\
							(\R \oplus \liealg{u}(n)) \ltimes \R^{2n} &\text{otherwise}.
		                                \end{cases}
	\end{equation*}
\end{corollary}
\begin{proof}
	By Aubin-Yau theorem we have an Einstein-K\"{a}hler metric which is unique up to homothety and $-c_{1}(M,J)$ is a Hodge class.
	Since $(M,J,g)$ is compact and simply connected with negative Einstein constant it is not
	symmetric and w.l.o.g. we have $\liealg{hol}(M,J,g) = \liealg{u}(n)$. Hence, the statement follows from
	Proposition $\ref{screenhoriz}$.
\end{proof}
Next, we construct Lorentzian manifolds such that $Hol(S,\nabla^{S}) = SU(n)$. In the following we say $(M,J,g)$ is
a Calabi-Yau manifold if $M$ is a compact K\"{a}hler manifold with $Hol(M,J,g)=SU(n)$.
Since $(M,J,g)$ is compact K\"{a}hler the Laplace operator commutes with the dual Lefschetz operator $\Lambda$
and we can define the primitive cohomology group
\begin{equation*}
	H^{1,1}_{prim}(M,J) := \mbox{Ker}(\Lambda: H^{1,1}(M,J) \longrightarrow \C).
\end{equation*}
Moreover, the Lefschetz decomposition implies
\begin{equation*}
	H^{1,1}((M,J),\R)=\R [\omega] \oplus H^{1,1}_{prim}((M,J),\R).
\end{equation*}
Let $\check{\Omega} = h \cdot \tilde{\Omega}$ for some nowhere-vanishing function $h \in C^{\infty}(X)$.
We say $(\tilde{J},\check{\Omega})$ defines an $SU(n)$-structure on $(S,\nabla^{S})$ if
$\nabla_{\cdot}^{S}\tilde{J}=\nabla_{\cdot}^{S}\check{\Omega}=0$.
Clearly, this implies $Hol(S,\nabla^{S}) \subset SU(n)$.
\begin{corollary}\label{screen-su}
	Let $(M,J,g)$ be a Calabi-Yau manifold. Let $(X=\tilde{X} \times L,\tilde{g}_{f})$ be of toric type over $(M,J,g)$ where
	$\tilde{X} \rightarrow M$ is constructed using a representative $\alpha$ of some $[\alpha] \in NS(M,J)$ and $f \in C^{\infty}(X)$
	is sufficiently generic. Suppose $\Lambda [\alpha] \in \Z$ or $L = \R$.
	Then $(\tilde{J},e^{-\sqrt{-1}(\Lambda \alpha)z}\tilde{\Omega})$ defines an $SU(n)$-structure on
	$(S,\nabla^{S})$ if and only if $\alpha$ is the harmonic representative of $[\alpha]$. In this case we have
	\begin{equation*}
		\liealg{hol}(X,\tilde{g}_{f}) = \begin{cases}
							\liealg{su}(n) \ltimes \R^{2n} &\text{if}~\frac{\partial f}{\partial x} \equiv 0,\\
							(\R \oplus \liealg{su}(n)) \ltimes \R^{2n} &\text{otherwise}.
		                                \end{cases}
	\end{equation*}
\end{corollary}
\begin{proof}
	If $[\alpha] \in NS(M,J)$ and $\alpha \in [\alpha]$ then $\Lambda [\alpha] = \mbox{const.}$ by the Lefschetz decomposition and
	$\nabla_{Z}^{S}(e^{-\sqrt{-1}(\Lambda \alpha)z}\tilde{\Omega})=0$ by Lemma $\ref{toric-holonomies}$ where $z$ is the coordinate on $L$.
	Moreover, for the harmonic representative of $[\alpha]$ we have $\nabla_{\cdot}^{S}(e^{-\sqrt{-1}(\Lambda \alpha)z}\tilde{\Omega})=0$.
	The converse is implied by the $\partial\bar{\partial}$-lemma and the K\"{a}hler identities.
\end{proof}
\begin{remark}\par \noindent
	\begin{enumerate}
		\item \noindent
		Let $(M,J,g)$ be a Calabi-Yau $n$-fold such that $n \geq 3$. In this case \cite{MR730926}[Prop. 2b]
		implies $h^{2,0}=h^{0,2}=0$. Assume that $[\omega] \in H^{2}(M,\Q)$ is its K\"{a}hler class. Then
		$\dim_{\Q}H^{1,1}_{prim}((M,J),\Q)=b_{2}-1$ (see \cite{MR0111056}), i.e., $rk(H^{1,1}_{prim}((M,J),\Z))=b_{2}-1$.
		Moreover, if $(X,\tilde{g}_{f})$ is as above and $\alpha \in [\alpha] \in H^{1,1}_{prim}((M,J),\Z)$ is harmonic
		then $(\tilde{J},\tilde{\Omega})$ defines an $SU(n)$-structure on $(S,\nabla^{S})$.
		\item \noindent
		If $(M,J,g)$ is an exceptional $K3$-surface\footnote{We call a $K3$-surface exceptional if its Picard number is maximal.}
		we can choose a basis $(b_{1},b_{2},[\omega],c_{1},\ldots,c_{19})$ of $H^{2}(M,\Q)$
		such that
		\begin{equation*}
			\mbox{span}_{\C}\{b_{1},b_{2}\} = H^{2,0}(M,J) \oplus H^{0,2}(M,J).
		\end{equation*}
		Using the intersection form and the Gram-Schmidt algorithm we derive a basis
		$\tilde{c}_{1},\ldots,\tilde{c}_{19}$ of $H^{1,1}_{prim}((M,J),\Q)$,i.e., $rk H^{1,1}_{prim}((M,J),\Z)=19$.
		Hence, the same remark applies for any harmonic $\alpha \in [\alpha] \in H^{1,1}_{prim}((M,J),\Z)$.
		\item \noindent
		Since $\liealg{su}(2)=\liealg{sp}(1)$ Corollary $\ref{screen-su}$ gives examples with symplectic screen holonomy.\qed
	\end{enumerate}
\end{remark}
Finally, we construct Lorentzian manifolds such that $Hol(S,\nabla^{S})=Sp(n)$. In the following a simply connected, compact K\"{a}hler manifold $X$
with $H^{2,0}(X)=\C [\sigma]$ where $\sigma$ is everywhere non-degenerate is called holomorphic symplectic. We write $\rho(X)$ for the
Picard number of a K\"{a}hler manifold $X$.
\begin{theorem}[K. Oguiso \cite{MR1966023}]\label{picard-hyperkaehler}
	Let $X$ be a holomorphic symplectic manifold with $b_{2}=N+2$. Then, for each integer $0 \leq k \leq N$ there exists a holomorphic symplectic
	manifold $X'$ such that $X$ and $X'$ are deformation equivalent and $\rho(X')=k$. \qed
\end{theorem}
Using Oguiso's theorem we can find a holomorphic symplectic structure with maximal Picard number on the differentiable manifold underlying
any holomorphic symplectic manifold.
\begin{corollary}
	Let $(M^{4n},J)$ be a holomorphic symplectic manifold with $b_{2} \geq 4$ and $\rho(M,J) = b_{2}-2$.
	Then, there exists an irreducible hyperk\"{a}hler structure $(M,J,J_{2},J_{3},g)$ with K\"{a}hler class $[\omega] \in H^{2}(M,\Q)$ and
	$0 \neq [\psi] \in H^{1,1}(M,J) \cap H^{1,1}(M,J_2) \cap H^{2}(M,\Z)$ on $M$. Moreover, if $(X=\tilde{X} \times L,\tilde{g}_{f})$ is of
	toric type over $(M,J,g)$ where $\tilde{X} \rightarrow M$ is constructed using the harmonic representative of $[\psi]$
	and $f \in C^{\infty}(X)$ is sufficiently generic then
	\begin{equation*}
		\liealg{hol}(X,\tilde{g}_{f}) = \begin{cases}
							\liealg{sp}(n) \ltimes \R^{4n} &\text{if}~\frac{\partial f}{\partial x} \equiv 0,\\
							(\R \oplus \liealg{sp}(n)) \ltimes \R^{4n} &\text{otherwise}.
		                                \end{cases}
	\end{equation*}
\end{corollary}
\begin{proof}
	We can find a K\"{a}hler class $[\omega] \in H^{2}(M,\Q)$ since $\rho(M,J)$ is maximal and Beauville's theorem \cite{MR730926}[Prop. 4.2]
	implies the existence of a hyperk\"{a}hler structure $(M,J=J_{1},J_{2},J_{3},g)$ on $M$ where $[\omega]$ is the K\"{a}hler class of
	$(M,J,g)$. We define operators
	\begin{equation*}
		\mbox{ad}J_{i} : \Lambda^{p,q}_{J_{i}}M \rightarrow \Lambda^{p+q}M \qquad \text{with} \qquad \eta \mapsto (p-q)\sqrt{-1}\eta.
	\end{equation*}
	It is shown in \cite{verbitsky-thesis}[Prop. 2.1] that the Lie algebra $\liealg{g}_{M}$ generated by
	$\mbox{ad}J_{1},\mbox{ad}J_{2},\mbox{ad}J_{3}$ is isomorphic to $\liealg{su}(2)$. Moreover, its action commutes with the Laplace operator and
	therefore induces an $\liealg{su}(2)$-action on the cohomology of $M$. In particular, \cite{verbitsky-thesis}[Prop. 5.2] implies
	$H_{inv} = H^{1,1}_{prim}(M)$ where $H_{inv}$ is the space of all $\liealg{g}_{M}$ invariant elements in $H^{2}(M)$. Finally, if $\alpha$ is
	the harmonic representant of some $[\alpha] \in H^{2}(M)$ then by definition $\alpha \in \Lambda^{1,1}_{J_{1}}M \cap \Lambda^{1,1}_{J_{2}}M$
	if and only if $[\alpha] \in H_{inv}$. Using the Hodge-Riemann bilinear form we derive a basis $c_{1},\ldots,c_{b_{2}-3} \in H^{2}(M,\Q)$ of
	$H^{1,1}_{prim}(M,J)$.
\end{proof}
\section{pp-waves and completeness}
Now we are in the position to give examples of non-trivial spaces realizing the holonomy algebras $\R \ltimes \R^{n}$ and $\R^{n}$.
As in \cite{MR2435292} we call an reducible indecomposable Lorentzian manifold $(X,\tilde{g})$ a $pr$-wave if $\liealg{hol}(X,\tilde{g})=\R \ltimes \R^{n}$
and a $pp$-wave if $\liealg{hol}(X,\tilde{g})=\R^{n}$. Lemma $\ref{toric-holonomies}$ implies
\begin{example}\label{pp-pr-wave-ex}
	Let $T^{n}= S^{1} \times \ldots \times S^{1}$ be the $n$-dimensional torus and $g$ the standard flat Riemannian metric on $T^{n}$. The standard
	coordinates induce a global trivialization $(\frac{\partial}{\partial y^{1}},\ldots,\frac{\partial}{\partial y^{n}})$ of $TT^{n}$ and $dy^{1} \wedge dy^{2}$
	is the volume form of $T^{2} \hookrightarrow T^{n}$. On $T^{n}$ we choose the $S^{1}$-bundle $\tilde{X} \rightarrow T^{n}$ defined by $cdy^{1} \wedge dy^{2}$.
	If $(X=\tilde{X} \times L,\tilde{g}_{f})$ is of toric type and $f \in C^{\infty}(X)$ is sufficiently generic then
	\begin{equation*}
		\liealg{hol}(X,\tilde{g}_{f}) = \begin{cases}
								\R^{n} & \text{if}~\frac{\partial f}{\partial x} \equiv 0,\\
								\R \ltimes \R^{n} & \text{otherwise}.
							\end{cases}
	\end{equation*}
	\qed
\end{example}
The only possible reducible indecomposable holonomy algebras in dimension $3$ are $\R$ and $\R \ltimes \R$. We conclude
\begin{example}\label{comp-tot-twist}
	Let $T^{2}$ be the flat torus with standard coordinates $(y^{1},y^{2})$. Define $X \rightarrow T^{2}$ using the volume form and Proposition \ref{mainconstr}
	with $\eta = dy^{2}$. If $f \in C^{\infty}(X)$ is sufficiently generic then
	\begin{equation*}
		\liealg{hol}(X,\tilde{g}_{f}) = \begin{cases}
								\R & \text{if}~\frac{\partial f}{\partial x} \equiv 0,\\
								\R \ltimes \R & \text{otherwise}.
							\end{cases}
	\end{equation*}
	In particular, $X$ is totally twisted.
\end{example}
\begin{proof}
	If $X = S^{1} \times Y$ then Gysin's sequence implies $b_{2}(X)= b_{2}(Y) + b_{1}(Y) = 2$ and $b_{1}(X)=1+b_{1}(Y)=2$, i.e.,
	$\chi(Y)=1$ and $b_{2}(Y)=1$. Finally, the classification of closed surfaces implies a contradiction.
\end{proof}
\begin{corollary}
	Let $T^{n+1}$ be the flat torus with $\psi := dy^{1}\wedge dz$ and $\eta := dz$ where $(y^{1},\ldots,y^{n},z)$ are the standard coordinates.
	If $(X,\tilde{g}_{f})$ is constructed as in Proposition \ref{mainconstr} with $f \in C^{\infty}(T^{n+1})$ sufficiently generic then
	$(X,\tilde{g}_{f})$ is a complete compact pp-wave.
\end{corollary}
\begin{proof}
	We have to show that the geodesics are defined for all $t \in \R$. Our approach is motivated by \cite{MR1971289}.
	Let $F_{n+1}: \R^{n+1} \rightarrow T^{n+1}$ be the universal covering map and consider the diagram
	\begin{equation*}
		\xymatrix{
			(\R \times \R^{n+1},(F_{n+1}^{*} \circ F_{1} \times \text{id})^{*}\tilde{g}_{f}) \ar[d]_{F_{1}\times \text{id}} & & \\
			(S^{1} \times \R^{n+1},F_{n+1}^{*}\tilde{g}_{f}) \ar[d]_{pr_{\R^{n+1}}} \ar[rr]^{F_{n+1}^{*}}
													& & (X,\tilde{g}_{f}) \ar[d]^{\pi}\\
			\R^{n+1} \ar[rr]^{F_{n+1}} & & T^{n+1}}
	\end{equation*}
	We write $g := (F_{n+1}^{*} \circ F_{1} \times \text{id})^{*}\tilde{g}_{f}$. Then
	\begin{equation*}
		g = 2dxdz + (y^{1} + f + 1)dz^{2} + \sum_{i=1}^{n}{(dy^{i})^{2}}.
	\end{equation*}
	Let $\gamma(t) = (x(t),y^{i}(t),z(t))$ be a curve on $\R^{n+2}$ of constant energy $E_{\gamma}:= g(\dot{\gamma},\dot{\gamma})$.
	We compute
	\begin{align*}
	0 &= \ddot{z} + \Gamma^{n+1}_{ij}\dot{x}^{i}\dot{x}^{j} = \ddot{z},\\
	0 &= \ddot{y}^{i} + \Gamma^{i}_{jk}\dot{x}^{j}\dot{x}^{k}
		= \ddot{y}^{i} + -\frac{\dot{z}^{2}}{2}(\frac{\partial f}{\partial y^{i}} + \delta^{i}_{1}),\\
	0 &= \ddot{x} + \Gamma^{0}_{ij}\dot{x}^{i}\dot{x}^{j}
		= \ddot{x} -\frac{1}{2}\frac{\partial f}{\partial z}\dot{z}^{2}
					+\sum_{i=1}^{n}{\frac{\partial f}{\partial y^{i}}\dot{z}\dot{y}^{i}}.
\end{align*}
Hence $\dot{z}=:A$ is constant. Let $\gamma_{2}$ be the projection of $\gamma$ to $\R^{n} \subset \R^{n+2}$ given by the $(y^{i})$ coordinates.
Then
\begin{equation}\label{eq-compl-rn}
	\frac{\nabla_{\skal{\cdot}{\cdot}}}{dt}\dot{\gamma}_{2} = \frac{A^{2}}{2}(\mbox{grad}_{\skal{\cdot}{\cdot}}f
											+ \frac{\partial}{\partial y^{1}})
							\qquad \text{on}~(\R^{n},\skal{\cdot}{\cdot}).
\end{equation}
Assume $\gamma_{2}$ is defined for all $t \in \R$. Since $E_{\gamma} = 2\dot{x}\dot{z} + (y^{1} + f + 1)\dot{z}^{2} + \sum_{i=1}{(\dot{y}^{i})^{2}}$
we conclude $x(t) = \dot{x}(0)t + x_{0}$ if $A=0$ and
\begin{equation*}
	x(t) = x_{0} + \frac{1}{2A}\int_{0}^{t}{E_{\gamma}-g(\dot{\gamma}_{2},\dot{\gamma}_{2}) -A^{2}(f(\gamma_{2}(s)) +1 + y^{1}(s))ds}
\end{equation*}
otherwise. In order to show the existence of $\gamma_{2}$ for all $t \in \R$ we define $\alpha(t):=(\gamma_{2},\dot{\gamma}_{2})$ and
\begin{equation*}
	F(x_{1},\ldots,x_{2n}) := (x_{n},\ldots,x_{2n},\frac{A^{2}}{2}(\partial_{1}f + 1),
									\frac{A^{2}}{2}\partial_{2}f,\ldots,\frac{A^{2}}{2}\partial_{n}f).
\end{equation*}
Then (\ref{eq-compl-rn}) is equivalent to $\dot{\alpha} = F(\alpha)$.
Let $C:=\sup_{T^{n+1}}{|f|}+\sup_{T^{n+1}}{|\nabla f|}$. If $\alpha$ is not defined for all $t \in \R$ then it must leave any compact
set. However, $\alpha(t)=\alpha(t_{0})+\int_{t_{0}}^{t}{F(\alpha(s))ds}$ and Gronwall's lemma imply that $\alpha$ is bounded on
any $[t_{0},t_{1})$ since
\begin{equation*}
	\|F(x)\|^{2} \leq \sum_{j=1}^{n}{(x_{j+n})^{2}} + \frac{A^{4}}{4}(C^{2}+2C+1) \leq \|x\|^{2} + \frac{A^{4}}{4}(C^{2}+2C+1).
\end{equation*}
Hence, $\gamma_{2}$ is defined for all $t \in \R$.
\end{proof}
\bibliographystyle{amsalpha}% or (ams)alpha
\nocite{*}
%    Insert the bibliography data here. %%Run latex then bibtex then latex%
\bibliography{sholconstr}
\end{document}